%% file: Arxivver.tex
\documentclass[12pt, reqno]{amsart}
\usepackage{hyperref}
\usepackage[utf8]{inputenc}
\usepackage{quiver}

\usepackage[mathscr]{euscript}
\usepackage{amsmath,amssymb,amsfonts,amsthm,mathrsfs}
\usepackage{tikz}
\usetikzlibrary{positioning}
\usepackage[margin=1in]{geometry}

\theoremstyle{definition}
\newcommand{\str}{\mathrm{str}}

\newcommand{\cU}{\mathcal{U}}
\input{canonicalheader}

\title[~]{Variation of Iwasawa Invariants for Ordinary Representations}

\author[Abhishek]{Abhishek}
\address[Abhishek]{Harish Chandra Research Institute, A CI of Homi Bhabha National Institute,  Chhatnag Road, Jhunsi, Prayagraj (Allahabad) 211 019 India}
\email{abhi.math04@gmail.com}
\author[C. Aribam]{Chandrakant Aribam}
\address[Aribam]{Indian Institute of Science Education and Research Mohali, Punjab (140306), India}
\email{aribam@iisermohali.ac.in}
\author[S. Barman]{Shiva Barman}
\address[Barman]{Indian Institute of Science Education and Research Mohali, Punjab (140306), India}
\email{shiva20barman@gmail.com}
\author[S. Ghosh]{Sohan Ghosh}
\address[Ghosh]{Indian Institute of Science, Bengaluru, Karnataka (560012), India}
\email{ghoshsohan4@gmail.com}

\begin{document}

\begin{abstract}
\noindent Let $K$ be a number field and $p$ be an odd prime.  Greenberg introduced a natural topology on the space $\mathcal{E}(K)$ of all $\Z_p$-extensions of $K$ and established several boundedness results for the classical Iwasawa invariants. We extend this framework to compare the Iwasawa invariants of Selmer groups attached to an ordinary $p$-adic representation across $\Z_p$-extensions lying in a Greenberg neighbourhood in the sense of \cite{Greenberg73}. We also establish analogous results for the fine Selmer groups. Finally, in a neighbourhood of the cyclotomic $\mathbb{Z}_p$-extension, we provide evidence for the expected connection between the characteristic ideal of the Selmer group and the conjectural $p$-adic $L$-function introduced by Disegni.
\end{abstract}
\maketitle
\section{Introduction}

Let $p$ be an odd prime and $K$ be a number field. Classical Iwasawa theory investigates the behaviour of various arithmetic invariants along infinite towers of number fields, more precisely along the $\mathbb{Z}_p$-extensions $K_{\infty}/K$. For each $n \geq 0$, let $K_n$ be the intermediate field satisfying $K \subset K_n \subset K_\infty$ with $\mathrm{Gal}(K_n/K) \cong \mathbb{Z}/p^n\mathbb{Z}$, and let $A_n$ denote the $p$-primary part of the class group of $K_n$. In this setting, Iwasawa proved the following fundamental theorem:
\begin{theorem}{\cite{Iwasawa59}}\label{Iwasawa's Thm}
    There exist integers $\lambda \geq 0$, $\mu \geq 0$ and $\nu$ such that for $n\gg 0$,
    \begin{center}
        $|A_n|=p^{\mu p^n+ \lambda n +\nu}$
    \end{center}
    The integers $\lambda, \mu, \nu$ are independent of $n$.
\end{theorem}
Iwasawa further conjectured that $\mu = 0$ in the case where $K_\infty = 
K_{\mathrm{cyc}}$ is the cyclotomic $\mathbb{Z}_p$-extension. This naturally 
raises the broader question of how the Iwasawa invariants behave across arbitrary 
$\mathbb{Z}_p$-extensions. Addressing this question, Greenberg \cite{Greenberg73} defined a topology on $\mathcal{E}(K)$, the space of all $\Z_p$-extensions of $K$, and proved several boundedness properties of the classical Iwasawa invariants. In particular, he showed that the $\mu$-invariant is locally bounded on $\mathcal{E}(K)$, while the $\lambda$-invariant is locally bounded in an open subset of $\calE(K)$ where the $\mu$-invariant vanishes. 
Moreover, if $K$ has a unique prime above $p$, then a constant depending on $K$ and $p$ bounds the $\mu$-invariants of all $\Z_p$-extensions of $K$.

A different type of boundedness result was later obtained by Fukuda \cite{Fukuda}. He proved that if the $p$-primary parts of the class groups stabilise at some finite layer, then they remain stable in all higher layers. In particular, if $|A_{n+1}|=|A_n|$ for some sufficiently large $n$, then $|A_m|= |A_n|$ for all $m\geq n$, and consequently $\mu=\lambda=0$.
Refining Greenberg's topology on the space of $\Z_p$-extensions of $K$ and generalising Fukuda's result, Kleine \cite{KleineCanad} established local boundedness results for the Iwasawa invariants of Selmer groups of abelian varieties. Under suitable hypotheses, he showed that there exists a sufficiently small neighbourhood $U$ of a given $\Z_p$-extension $K_\infty/K$ such that, for every $\widetilde{K}_\infty \in U$, the Selmer group $\Sel(A/\widetilde{K}_\infty)$ remains cotorsion, its $\mu$-invariant is bounded above by that of $\Sel(A/K_\infty)$, and its $\lambda$-invariant is bounded above by that of $\Sel(A/K_\infty)$ whenever the corresponding $\mu$-invariants agree.
Building on the ideas of Greenberg and Fukuda, we extend the work of Kleine \cite{KleineCanad} to compare the Iwasawa invariants for Selmer groups (respectively, fine Selmer groups) associated with $p$-ordinary two-dimensional Galois representations over various $\mathbb{Z}_p$-extensions of a number field $K$. 

Let $\calK$ be a finite extension of $\Q_p$ with ring of integers $\calO_\calK$. Let $V\cong \calK^2$ be a two-dimensional $p$-ordinary representation of $G_K$ and fix a $G_K$-stable $\calO_\calK$-lattice $T$. Set $A={V}/{T}$. For a discrete $p$-primary module $M$, write $M^\vee:=\Hom_{\mathrm{cont}}(M,\Q_p/\Z_p)$ for its Pontryagin dual. We denote the $p$-adic cyclotomic character by $\omega_p$.   Let $\Sel^{\Gr}(A/K_\infty)$ (defined in \S \ref{3.1}) denote the Greenberg Selmer group. Our first main result is the following theorem:
\begin{theorem}[Theorem \ref{prop:SelFukuda}]
 Let $V$ be a two-dimensional $p$-ordinary representation of $G_K$. Let $K_\infty$ be a $\mathbb{Z}_p$-extension of $K$. Suppose that, for every prime $v\mid p$, the inertia group $I_v$ acts on $A^-$ via $\omega_p^{-k_v}$ for some $k_v\in \Z_{>0}$, and that $H^0(G_K,A)=0$. Assume also that, for every $v\in P(K_\infty)$ and every prime $w\mid v$ of $K_\infty$, the group $H^0(I_{\infty,w},A^-)$ is finite.
If $\Sel^{\Gr}(A/K_\infty)^\vee$ is a torsion $\Lambda$-module, then there exists $r\geq0$ such that, with $U=\mathcal{U}(K_\infty,r)$,
     \begin{enumerate}[(a)]
         \item $\Sel^{\Gr}(A/\widetilde{K}_\infty)^\vee$ is a torsion $\Lambda$-module for each $\widetilde{K}_\infty\in U$,
         \item $\mu(\Sel^{\Gr}(A/\widetilde{K}_\infty)^\vee)\leq \mu(\Sel^{\Gr}(A/K_\infty)^\vee)$ for each $\widetilde{K}_\infty\in U$,
         \item $\lambda(\Sel^{\Gr}(A/\widetilde{K}_\infty)^\vee)\leq \lambda(\Sel^{\Gr}(A/K_\infty)^\vee)$ for each $\widetilde{K}_\infty\in U$ whenever $\mu(\Sel^{\Gr}(A/\widetilde{K}_\infty)^\vee)=\mu(\Sel^{\Gr}(A/K_\infty)^\vee)$.
     \end{enumerate}
\end{theorem}
We also prove a similar result for the strict Selmer groups. Whereas Kleine's results concern Selmer groups of ordinary abelian varieties, our framework applies to ordinary two-dimensional Galois representations and also treats strict, classical fine, and the Greenberg fine Selmer groups.

The second theme in this article is to study the Iwasawa invariants of the fine Selmer groups. Coates--Sujatha \cite{CS05} formally defined the fine Selmer group of an elliptic curve. They observed an important relation between the structure of $R(E/K_{\cyc})$ and Iwasawa's $\mu=0$ conjecture for $K_{\cyc}$.
\begin{theorem}[\cite{CS05},Theorem 3.4]
     Let $E/K$ be an elliptic curve and $p$ be an odd prime such that $K(E_{p^{\infty}})/K$ is a pro-$p$ extension. Then $R(E/K_{\cyc})^{\vee}$ is a finitely generated $\Z_p$ module if and only if Iwasawa's $\mu=0$ conjecture holds for $K_{\cyc}$.
\end{theorem}

Fine Selmer groups have been studied extensively by Greenberg, Kurihara, and others. Kleine also studied their growth for abelian varieties with good ordinary reduction over $\mathbb{Z}_p$-extensions of $K$. In this article, we consider two types of fine Selmer groups. One of them is the usual fine Selmer group defined by Coates--Sujatha \cite{CS05}, and the other we call the Greenberg fine Selmer group (defined in \ref{defn:fineSelmer}). For a general two-dimensional Galois representation, the fine Selmer group depends on the choice of $S$, whereas the Greenberg fine Selmer group does not; for example, fine Selmer groups over function fields of characteristic $p$ need not be independent of $S$ \cite{GJS}.

For the classical and Greenberg fine Selmer groups, we establish uniform boundedness results for the associated Iwasawa invariants in neighbourhoods of $\Z_p$-extensions (Theorems \ref{Thm fine sel_1} and \ref{Gr fine selfukuda}).

The third theme of this article concerns the Iwasawa Main Conjecture, which predicts that the characteristic ideal of the dual Selmer group is generated by the associated $p$-adic $L$-function. Although the Iwasawa Main Conjecture has been studied extensively over the cyclotomic $\mathbb{Z}_p$-extension, it is natural to ask whether this relationship persists across nearby $\mathbb{Z}_p$-extensions. Building on the construction of $p$-adic $L$-functions over arbitrary $\mathbb{Z}_p$-extensions due to Disegni \cite{Disegni}, we show that, assuming the Iwasawa Main Conjecture over $K_\cyc$, the expected relation propagates to a Greenberg neighbourhood of the cyclotomic $\mathbb{Z}_p$-extension. Let $E/K$ be an elliptic curve with good ordinary reduction at prime $p$. For a $\Z_p$-extension $K_\infty$ with Galois group $\Gamma$, let $L_p^{(\Gamma)}(E)$ be the conjectural $p$-adic $L$-function defined in \cite{Disegni} satisfying the hypothesis \hyperref[Hypothesis]{$(L_p)$}. 

\begin{theorem}(Theorem \ref{thm:example})
Let $E/K$ be an elliptic curve with good ordinary reduction at $p$, and suppose that $\Sel^{\Gr}(E/K_\cyc)^\vee$ is a finitely generated torsion $\Lambda$-module with $\mu(\Sel^{\Gr}(E/K_\cyc)^\vee)=\lambda(\Sel^{\Gr}(E/K_\cyc)^\vee)=0$. Assume that the Iwasawa Main Conjecture holds for $K_\cyc$. Then there exists a neighbourhood $\calU(K_\cyc,m)$ of $K_\cyc$ such that, for every $\widetilde{K}_\infty\in \calU(K_\cyc,m)$, the characteristic ideal of $\Sel^{\Gr}(E/\widetilde{K}_\infty)^\vee$ is generated by the $p$-adic $L$-function $L^{\widetilde{\Gamma}}_\gp(E)$.
\end{theorem}

\noindent We conclude with a numerical example (Example \ref{ex:IMC}) illustrating the setting of the theorem.

\subsection*{Organisation of the article}
The paper is organised as follows. Section~\ref{preli} recalls the required definitions and algebraic preliminaries. Section~\ref{Results on Sel} proves the control and boundedness results for Selmer and fine Selmer groups. Section~\ref{example} discusses the conjectural $p$-adic $L$-functions and the application to the Iwasawa Main Conjecture.
\section{Preliminaries}\label{preli}
In this section, we introduce the main definitions and recall the preliminaries needed for the subsequent sections. For a field $L$, we denote $\Gal( \overline{L}/L)$ by $G_L$ where $\overline{L}$ is the separable closure of $L$. Throughout we fix an algebraic closure $\overline{\Q}$ of $\Q$ in $\mathbb{C}$ and  embeddings $\iota_\infty:\overline{\Q}\rightarrow\mathbb{C}$ and  $\iota_\ell:\overline{\Q}\rightarrow\overline{\Q}_\ell$, for every finite rational prime $\ell$. 

Fix a prime $p\geq 3$ and let $K$ be a number field. Let $\calK$ be a finite extension of $\Q_p$ and $\calO_{\calK}$ be its ring of integers.  Let $T$ be a free  $\calO_{\calK}$-module of finite rank with a continuous linear action of $G_K$. Set
$$
V:=T\otimes_{\calO_{\calK}}\calK,\qquad
A:=V/T,\qquad
T^{*}:=\Hom(A,\mu_{p^\infty}),\qquad
V^{*}:=T^{*}\otimes_{\calO_{\calK}}\calK.
$$
For a prime $v$ of $K$, write $G_v:=\Gal(\overline{K}_v/K_v)$ and let $I_v\subset G_v$ be the inertia subgroup.

\noindent
\begin{definition}\cite{Greenberg89}
%\textbf{Definition (Ordinarity) \cite{Greenberg89}}
Let $v\mid p$ be a prime of $K$. We say that a $G_K$ representation $V$ is ordinary at $v$ if there exists a decreasing filtration $\{F^i(V)\}_{i\in\Z}$ by $\calK$-subspaces satisfying the following conditions:
\begin{enumerate}[\textbullet]
    
    \item  There exist integers $j_1 \le j_2$ such that
    \[
    F^i(V)=V \quad \text{for all } i\le j_1, \qquad
    F^i(V)=0 \quad \text{for all } i\ge j_2.
    \]
    \item Each $F^i(V)$ is $G_v$-stable, and $I_v$ acts on the graded piece
    $
      \gr^i(V):=F^i(V)/F^{i+1}(V)
    $
    via the $i$-th power of the cyclotomic character.
    \item Let $V^+ := F^{1}(V)$ (equivalently, $V^+=\bigcup_{i>0} F^i(V)$). Then $V^+$ inherits a filtration whose graded pieces are acted on by $I_v$ through positive powers of the cyclotomic character. Moreover, there is a short exact sequence of $G_v$-representations
    $$
      0\longrightarrow V^+\longrightarrow V\longrightarrow V^-\longrightarrow 0
    $$
    in which $V^-$ admits a filtration whose graded pieces carry the $I_v$-action via non-positive powers of the cyclotomic character. The filtration above is called the \emph{Panchishkin filtration} of $V$.
\end{enumerate}
\end{definition}

Define the associated lattices and quotients by
$$
T^+:=T\cap V^+,\qquad
T^-:=T/T^+,\qquad
A^+:=V^+/T^+,\qquad
A^-:=V^-/T^-.
$$

We make the following standing assumption:

\begin{assumption}\label{ass:global}
The representation $V$ is ordinary at every prime $v\mid p$ and unramified outside finitely many primes of $K$.
\end{assumption}

%\begin{remark}\label{rk_1}
%%Let $\rho$ be the Galois representation attached to a newform $f\in S_k(\Gamma_0(Np^t))$ which is ordinary at $p$. for a prime $v\mid p$, the action of $I_v$ on $A^-$ is via the cyclotomic character $\omega_p$. Then, as given in \cite{longovigni}, we can show that if there is a $\sigma\in I_{\infty, w}$ such that $\omega_p(\sigma)\ne 1$, then $H^0(I_{\infty, w}, A^-)=0$ and consequently $H^0(I_v, A^-)=0$.  
%\end{remark}}

\subsection{Selmer groups}\label{3.1}
Let $F$ be a finite extension of $K$. For a prime $v$ of $K$ define
 \[
 J^{\Gr}_v(A/F):=\begin{cases}
     \underset{w\mid v}\prod \frac{H^1(F_w,A)}{H^1_{\unram}(F_w,A)} \text{ if } v\nmid p, \\
     \underset{w\mid v}\prod \frac{H^1(F_w,A)}{H^1_{\ord}(F_w,A)} \text{ if } v\mid p,
\end{cases}
 \]
where
\[
H^1_\unram(F_w,A):=\ker \left(H^1(F_w,A)\lra H^1(I_w,A)   \right)\ \text{and}\ H^1_{\ord}(F_w,A):=\ker \left(H^1(F_w,A)\lra H^1(I_w, A^-) \right). 
\]

For a prime $v$ of $K$, define
\[
J^{\str}_v(A/F):=\begin{cases}
    \underset{w\mid v}\prod H^1(F_w,A) & \text{if } v\nmid p, \\
    \underset{w\mid v}\prod H^1(F_w,A^-) & \text{if } v\mid p.
\end{cases}
\]

\begin{definition}\label{defn:strict}
For $*\in\{\Gr,\str\}$, define the corresponding Selmer group of $A$ over $F$ by
\[
\Sel^{*}(A/F):=\ker\left(H^1(F,A)\lra \underset{v}\prod J^{*}_v(A/F)\right).
\]
Here, $v$ ranges over all primes of $K$.
\end{definition}

\subsection{Fine Selmer groups}\label{defn:fineSelmer}
We define two types of fine Selmer groups here. The first is the usual fine Selmer group defined in \cite{CS05}, and the second one we call the Greenberg fine Selmer group.
These fine Selmer groups are the subgroups of the usual Selmer groups.
Let $S$ be a finite set of places of $K$ containing all places above $p$, all archimedean places, and every finite place at which $V$ is ramified. Let $F_S$ be the maximal extension of $F$ unramified outside the places above $S$, and put $G_S(F):=\gal(F_S/F)$.  
\begin{definition}
    The $S$-fine Selmer group of $A$ over $F$ is defined as:
    \begin{align*}
    R_S(A/F):= \ker\left( H^1(G_S(F),A) \longrightarrow \underset{v\in S}\bigoplus P^1_v(A/F)\right), 
\end{align*}  
where $P^1_v(A/F):=\underset{w\mid v}\prod H^1(F_w, A). $ 
\end{definition}
\begin{definition}\label{defn:gr fine sel}
   The Greenberg fine Selmer group of $A$ over $F$ is defined as
\begin{align*}
    R^{\Gr}(A/F):= \ker\left( H^1(F,A) \longrightarrow \underset{v}\prod K^1_v(A/F)\right),
\end{align*}
where
\[
 K^1_v(A/F):=\begin{cases}
     \underset{w\mid v}\prod \frac{H^1(F_w,A)}{H^1_{\unram}(F_w,A)} \text{ if } v\nmid p, \\
     \underset{w\mid v}\prod H^1(F_w,A) \text{ if } v\mid p.
\end{cases}
 \]
 Here, the product runs over all the primes of $K$.
\end{definition}

Over infinite algebraic extensions, the definitions of the Selmer groups and the fine Selmer groups extend naturally by taking inductive limits.
Note that, in many cases, $R_S(A/F)$ coincides with $R^{\Gr}(A/F)$ (see, for instance, Remarks \ref{rem:3.12}, and \ref{rem:3.13}).

\subsection{Fukuda modules}\label{subsec:Fukudamodules}
In \cite{Kleine}, Kleine introduced a special family of Iwasawa modules $X = \varprojlim X_n$. These modules were designed to derive information about the projective limit by utilising data from a sufficiently large number of layers $X_n$.
\begin{definition}\label{defn:Fukudamodules}
Let $d\ge1$ be an integer, let $R:=\Z_p[[T_1,\ldots,T_d]]$, and let $\mathfrak m=(p,T_1,\ldots,T_d)$. Let $(X_n)_{n\in\N}$ be a projective system of finite $R$-modules of $p$-power order, put $X=\varprojlim_n X_n$, and assume that $X$ is compact for the $\mathfrak m$-adic topology. Write $\pr_n:X\to X_n$ for the projection and $Y_n=\ker(\pr_n)$.

Let $C_1,C_2,C_3$ be powers of $p$. We call $X$ a Fukuda $R$-module with parameters $(C_1,C_2,C_3)$ if there are compact $R$-submodules $Z_n\subset X$ such that, for every $n\in\N$,
\begin{enumerate}[(i)]
\item $|\coker(\pr_n)|\le C_1$;
\item $Z_{n+1}\subset \mathfrak m Z_n$;
\item $[Y_n:Y_n\cap Z_n]\le C_2$;
\item $[Z_n:Y_n\cap Z_n]\le C_3$.
\end{enumerate}
\end{definition}

Consider a $\mathbb{Z}_p^d$-extension $K_\infty/K$, where $d \geq 1$. Write $K_\infty = \underset{n}\bigcup K_n$ with $\Gal(K_n/K) \cong (\Z/p^n\Z)^d$. Let $X = \varprojlim X_n$, where $X_n$ is the $p$-primary subgroup of the ideal class group of $K_n$. Then, under appropriate assumptions on the ramification of primes in $K_\infty/K$,  one can show that $X$ is a Fukuda $R$-module with parameters $(1,1,1)$ (see \cite[Section \S 3]{KleineCanad}), where $R \cong \mathbb{Z}_p[[\Gal(K_\infty/K)]]$.

\begin{remark}
        The definition of Fukuda modules can be weakened such that the modules $Z_n$ satisfy the conditions in Definition \ref{defn:Fukudamodules} only for $n \geq e$ for some integer $e$ (cf. \cite[Remark 3.3]{KleineCanad}). In particular, the proof of \cite[Theorem 4.11]{KleineCanad} remains unchanged. We use this fact to generalise the result for general ordinary representations.  
\end{remark}

\subsection{Greenberg neighbourhoods}\label{Greenberg neighbourhoods}
Let $\mathcal{E}(K)$ denote the set of all $\mathbb{Z}_p$-extensions of $K$. 
Greenberg introduced a natural topology on $\mathcal{E}(K)$, defined as follows: 
for $L_\infty \in \mathcal{E}(K)$ and $n\in\mathbb{N}$, set
\[
\mathcal{E}(L_\infty,n)
   :=\{\,L_\infty' \in \mathcal{E}(K)
          \mid [L_\infty \cap L_\infty' : K] \ge p^n\,\}.
\]
The topology on $\mathcal{E}(K)$ is generated by the collection of all such sets 
$\mathcal{E}(L_\infty,n)$, with $L_\infty \in \mathcal{E}(K)$ and $n \in \mathbb{N}$.
Kleine later introduced a finer topology on $\mathcal{E}(K)$.
\begin{definition}[\cite{Kleine}]
For each $\mathbb{Z}_p$-extension $L_\infty/K$ and every $n\in\mathbb{N}$, define
\[
\cU(L_\infty,n)
   :=\{\,M_\infty \in \mathcal{E}(L_\infty,n)
          \mid P(M_\infty)\subset P(L_\infty)\,\},
\]
where $P(L_\infty)$ (resp.\ $P(M_\infty)$) denotes the set of primes of $K$ 
that ramify in $L_\infty/K$ (resp.\ $M_\infty/K$). Note that primes of $K$ lying above $p$ are the only primes that can occur in this set. We consider the topology on $\calE(K)$ that is generated by the sets $\calU(L_\infty,n)$. 
\end{definition}
\subsection{Iwasawa algebra}
Let $K_\infty\in \mathcal{E}(K)$ be  a $\Z_p$-extension of $K$, with $\Gamma:=\Gal(K_\infty/K)\cong \Z_p$. For each integer $n\ge0$, let $K_n$ be the subfield of $K_\infty$ satisfying $\Gal(K_n/K)\cong\Z/p^n\Z$, and set $\Gamma_n:=\Gal(K_\infty/K_n)$. Let $\Lambda$ denote the Iwasawa algebra $\Z_p[[\Gamma]]\cong \Z_p[[T]]$.
{ Let $M$ be a  finitely generated $\Lambda$-module. Then, by the structure theorem of finitely generated Iwasawa modules, there exists {a homomorphism of $\Lambda$-modules} 

\[ M \lra \Lambda^{r_M} \oplus \bigoplus_{i=1}^{s_M} \frac{\Lambda}{(f_i^{m_i})} \oplus \bigoplus_{j=1}^{t_M} \frac{\Lambda}{(p^{n_j})} \]
with a finite kernel and cokernel. Here $r_M,\ s_M,\ t_M,\ m_i,\ n_j $ are non-negative integers, which {depend} on $M$, and  $f_i$'s are distinguished irreducible polynomials in $\Lambda$ (see \cite[Chapter 5,\S3]{NSW} for definition). 
Furthermore, if $M$ is a  torsion $\Lambda$-module, we define the \emph{characteristic ideal} as:
\begin{align*}
    \mathrm{char}_\Lambda(M):=(\underset{j=1}{\overset{t_M}\prod}p^{n_j}.\underset{i=1}{\overset{s_M}\prod}f_i^{m_i} ),
\end{align*}
where the notation is as above. We define
\[
\mu(M):=\sum_{j=1}^{t_M}n_j,
\qquad
\lambda(M):=\sum_{i=1}^{s_M}m_i\deg(f_i).
\]

\section{Main results}{\label{Results on Sel}}
Let $K$ be a number field, and let $V$ be a two-dimensional representation of $G_K$ that is ordinary at every prime of $K$ above $p$ and unramified outside a finite set of primes. Let $T$ and $A$ be as in Section \S\ref{preli}. We first establish a control theorem for the associated Selmer groups over an arbitrary $\Z_p$-extension $K_\infty/K$.

\begin{proposition}\label{thm:controlSel}
Let $V$ be a two-dimensional representation of $G_K$ which is ordinary at all the primes above $p$.
Let $K_\infty$ be a $\mathbb{Z}_p$-extension of $K$. Suppose that

\begin{enumerate}[(i)]
    \item for every prime $v\mid p$, the action of $I_v$ on $A^-$ is via a negative power of the cyclotomic character.
    \item $H^0(G_K, A) = 0$.
\end{enumerate}

 Then the kernel of the restriction map 
\[
\Sel^{\mathrm{Gr}}(A/K_n)
   \overset{f_n}\longrightarrow
   \Sel^{\mathrm{Gr}}(A/K_\infty)^{\Gamma_n}
\]
is zero for every $n$. Furthermore, there exists $m>0$ such that $\coker(f_n)$ is finite and bounded independently of $n$ for all $n\ge m$. 
\end{proposition}
\begin{proof}

    Let $K\subset K_n\subset K_\infty$, such that $\Gamma_n:=\Gal(K_\infty/K_n)$. Now consider the following commutative diagram:

\begin{equation}
\begin{tikzcd}
	0 & {\Sel^{\Gr}(A/K_{\infty})^{\Gamma_n}} & {H^1(K_{\infty}, A)^{\Gamma_n}} &  {\underset{v}{\prod}} J^{\Gr}_v(A/K_\infty)^{\Gamma_{n,v_n}} \\
	0 & {\Sel^{\Gr}(A/K_n)} & {H^1(K_{n}, A)} & {\underset{v}\prod J^{\Gr}_v(A/K_n)}
	\arrow[from=1-1, to=1-2]
	\arrow[from=1-2, to=1-3]
	\arrow[from=1-3, to=1-4]
	\arrow[from=2-1, to=2-2]
	\arrow["f_n"', from=2-2, to=1-2]
	\arrow[from=2-2, to=2-3]
	\arrow["g_n"', from=2-3, to=1-3]
	\arrow[from=2-3, to=2-4]
	\arrow["h_n=\underset{v}\prod h_{n,v_n}"', from=2-4, to=1-4]
\end{tikzcd}
\end{equation}
By applying the snake lemma, it suffices to show that the kernels and cokernels of the map $g_n$ vanish for all $n$, and that the kernel of the map $h_n$ is finite and uniformly bounded for all $n \ge m$. 
Since $H^0(G_{K},A)=0$,  the Nakayama lemma gives us  that $H^0(G_{K_\infty},A)=0$ and  hence $\ker g_n=H^1(\Gamma_n, A^{G_{K_\infty}})=0$. Moreover, since $\Gamma_n$ has $p$-cohomological dimension $1$, we also have $\coker(g_n)=0$ for all $n$.   

\iffalse
For $\ker(g_n)$, an analogous argument as in  \cite[Lemma 3.1]{Greenberg97}, shows that it is finite and of order bounded as $n$ varies. The proof in \emph{loc. cit} is for Galois representations coming from elliptic curves, and the same argument works here in our situation.
\fi

We proceed to analyse the map $\ker(h_n)$.
Suppose that $v_n$ is a prime of $K_n$ ramified in $K_\infty$.
Then, by \cite[Lemma 13.3]{washingtonbook} $v_n \mid p$. We claim that $\ker(h_{n,v_n})$ is finite. Now, choose a prime $w\mid v_n$ of $K_\infty$ and put $\Gamma_{n,v_n}:=\gal(K_{\infty,w}/K_{n,v_n})$. Consider the following commutative diagram with exact rows,
\begin{equation}
    \begin{tikzcd}
	0 & {H^1_{\ord}(K_{\infty,w},A)}^{\Gamma_{n,v_n}} & {H^1(K_{\infty,w},A)}^{\Gamma_{n,v_n}} & {H^1(I_{\infty,w},A^-)}^{\Gamma_{n,v_n}} \\
	0 & {H^1_{\ord}(K_{n,v_n},A)} & {H^1(K_{n,v_n},A)} & {H^1(I_{n,v_n}, A^-)}.
    \arrow[from=1-1, to=1-2]
	\arrow[from=1-2, to=1-3]
	\arrow[from=1-3, to=1-4]
	\arrow[from=2-1, to=2-2]
	\arrow[from=2-2, to=1-2]
	\arrow[from=2-2, to=2-3]
	\arrow[from=2-3, to=1-3]
    \arrow[from=2-3, to=2-4]
    \arrow[from=2-4, to=1-4]
\end{tikzcd}
\end{equation}
The kernel of the rightmost vertical map is $H^1(I_{n,v_n}/I_{\infty,w}, (A^-)^{I_{\infty,w}})$. 
Now, we consider two possible cases:

\underline{Case 1:} Suppose $I_{\infty,w}$  acts non-trivially on  $A^-$. Therefore  $H^0(I_{\infty,w}, A^-)$ is finite. Let $B=H^0(I_{\infty,w}, A^-)$ and note that $H^1(\Gamma_{n,v_n},B)\cong \dfrac{B}{(\gamma-1)B}$, where $\gamma$ is the topological generator of $\Gamma_{n,v_n}$. This shows that $\ker h_{n, v_n}=H^1(I_{n,v_n}/I_{\infty,w}, (A^-)^{I_{\infty,w}})$ has order bounded above by $|H^0(I_{\infty,w},A^-)|$, and hence is uniformly bounded.

\underline{Case 2:} Suppose $I_{\infty,w}$ acts trivially on $A^-$, so $(A^-)^{I_{\infty,w}}=A^-$. Let $\gamma$ be a topological generator of $I_{n,v_n}/I_{\infty,w}$. By hypothesis, the action on the one-dimensional space $V^-$ is nontrivial; hence $\gamma-1$ is invertible and
\[
H^1(I_{n,v_n}/I_{\infty,w},V^-)
   \cong V^-/(\gamma-1)V^-=0.
\]
The long exact cohomology sequence attached to $0\to T^-\to V^-\to A^-\to0$, together with the fact that $I_{n,v_n}/I_{\infty,w}$ has $p$-cohomological dimension one, gives
$H^1(I_{n,v_n}/I_{\infty,w},A^-)=0$. Thus $\ker(h_{n,v_n})=0$ in this case.

Let $v_n$ be a prime of $K_n$ unramified in $K_\infty$.
Choose a prime $w \mid v_n$ of $K_{\infty}$ and put $\Gamma_{n,v_n}:=\gal(K_{\infty,w}/K_{n,v_n})$. Since $K_{\infty,w}/K_{n,v_n}$ is unramified, we have $I_{n,v_n}=I_{\infty,w}$, so $\ker(h_{n,v_n})=0$ for $v_n\nmid p$.
\iffalse
\begin{align*}
    H^1(I_{n,v_n}/I_{\infty,v_n},A^{I_{\infty,v_n}})=0.
\end{align*}

Now, consider the following commutative diagram:
\[\begin{tikzcd}
	0 & {H^1_{\unram}(K_{\infty,w},A)^{\Gamma_{n,v_n}}} & {H^1(K_{\infty,w},A)^{\Gamma_{n,v_n}}} & {H^1(I_{\infty,w},A)^{\Gamma_{n,v_n}}} \\
	0 & {H^1_{\unram}(K_{n,v_n},A)} & {H^1(K_{n,v_n},A)} & {H^1(I_{n,v_n},A)} & 0
	\arrow[from=1-1, to=1-2]
	\arrow[from=1-2, to=1-3]
	\arrow[from=1-3, to=1-4]
	\arrow[from=2-1, to=2-2]
	\arrow[from=2-2, to=1-2]
	\arrow[from=2-2, to=2-3]
	\arrow[from=2-3, to=1-3]
	\arrow[from=2-3, to=2-4]
	\arrow[from=2-4, to=1-4]
	\arrow[from=2-4, to=2-5]
\end{tikzcd}\]
\fi
When {Case(1)} occurs, we choose $m=1$, otherwise, we choose $m$ to be the smallest integer such that every prime $v_m$ of $K_m$ ramified in $K_{\infty}$ is totally ramified.
By Shapiro's lemma, we have that $\ker(h_n)$ is finite for every $n \geq m$. This completes the proof of the proposition.
\end{proof}
\begin{remark}
   Let $f$ be a newform of even weight $k \geq 4$ and level $\Gamma_0(N)$, with $p\nmid N$. Let $\Q_f$ denote the Hecke field of $f$, and fix a prime $\mathfrak{p}\mid p$ of $\Q_f$. Assume moreover that $f$ is $p$-ordinary and satisfies
$a_p(f)\not\equiv 1 \pmod{\mathfrak{p}}.$
Then, by \cite[pp.~444--445]{longovigni}, one has
$H^0(K_v,A)=0$
for every prime $v\mid p$ of $K$. Consequently, $H^0(G_K,A)=0,$
and therefore assumption~(ii) of Proposition~\ref{thm:controlSel} is satisfied in this setting.
\end{remark}
\begin{corollary}
Let $K_\infty$ be a $\mathbb{Z}_p$-extension of $K$.
Assume the hypotheses of Proposition  \ref{thm:controlSel}. Then $\Sel^{\Gr}(A/K_\infty)^\vee$
is a Fukuda-$\Lambda$-module with parameters $(C_1,C_2,1)$
for suitable $p$-power integers $C_1$ and $C_2.$
\end{corollary}
\begin{proof}
Taking $X_n^{(K_{\infty})}:=\Sel^{\Gr}(A/K_n)^{\vee}$ and $X^{(K_{\infty})}:=\Sel^{\Gr}(A/K_{\infty})^{\vee}$, and considering the natural projection map $\mathrm{pr}_n:X^{(K_{\infty})}\to X_n^{(K_{\infty})}$, the proof proceeds by arguments analogous to \cite[Corollary 4.2]{KleineCanad}.

\end{proof}
\noindent The following Lemma is a generalisation of \cite[Theorem 4.5]{KleineCanad}.
\begin{lemma}\label{thm:Fukudanbd}
   Let $K$ be a number field and $K_\infty$ be a $\Z_p$-extension of $K$. Let us assume the hypotheses of Proposition  \ref{thm:controlSel}. Further, suppose that, for every $v\in P(K_\infty)$ and every prime $w\mid v$ of $K_\infty$, the group $H^0(I_{\infty,w},A^-)$ is finite.
Then there exist an integer $r\ge1$ and powers of $p$, $C_1,C_2\ge1$, such that $\Sel^{\Gr}(A/\widetilde{K}_\infty)^{\vee}$ is a Fukuda $\Lambda$-module with bounded parameters $(C_1, C_2, 1)$ for $\widetilde{K}_\infty\in \calU(K_\infty,r)$.
\end{lemma}
\begin{proof}
Following the strategy of \cite[Theorem 4.5]{KleineCanad}, it is enough to show that the kernel and cokernel of  the map $\Sel^{\Gr}(A/\widetilde{K}_n)\overset{f_n^{\widetilde{K}_{\infty}}}\lra \Sel^{\Gr}(A/\widetilde{K}_\infty)^{\Gamma_n}
$ is finite and bounded independently of $n$ for every $n>m$ and every $\widetilde{K}_\infty\in \calU(K_\infty,r)$ for suitable choices of $r$ and $m$.
For each $n>0$, we consider the following commutative diagram:
\begin{equation}
\begin{tikzcd}
	0 & {\Sel^{\Gr}(A/\widetilde{K}_{\infty})^{\Gamma_n}} & {H^1(\widetilde{K}_{\infty}, A)^{\Gamma_n}} &  {\underset{v}{\prod}} J^{\Gr}_v(A/\widetilde{K}_\infty)^{\Gamma_{n,v_n}} \\
	0 & {\Sel^{\Gr}(A/\widetilde{K}_n)} & {H^1(\widetilde{K}_{n}, A)} & {\underset{v}\prod J^{\Gr}_v(A/\widetilde{K}_n)}
	\arrow[from=1-1, to=1-2]
	\arrow[from=1-2, to=1-3]
	\arrow[from=1-3, to=1-4]
	\arrow[from=2-1, to=2-2]
	\arrow["f_n^{\widetilde{K}_{\infty}}"', from=2-2, to=1-2]
	\arrow[from=2-2, to=2-3]
	\arrow["g_n^{\widetilde{K}_{\infty}}"', from=2-3, to=1-3]
	\arrow[from=2-3, to=2-4]
	\arrow["h_n^{\widetilde{K}_{\infty}}=\underset{v}\prod h_{n,v_n}^{\widetilde{K}_{\infty}}"', from=2-4, to=1-4]
\end{tikzcd}
\end{equation}
From the proof of Proposition \ref{thm:controlSel} we see that $\ker(f_n^{\widetilde{K}_{\infty}})=0$ for all $n$. 
Next, we give a uniform bound for $\ker(h_{n,v_n}^{\widetilde{K}_{\infty}})$ for each $v_n$. If $v_n$ is a prime of $K_n$ unramified in $\widetilde{K}_\infty$, then arguing as in the proof of Proposition \ref{thm:controlSel}, we get that $\ker{h_{n,v_n}^{\widetilde{K}_{\infty}}}=0$. Note that if $\widetilde{K}_\infty \in \mathcal U(K_\infty,r)$ for some $r>0$, then any prime $v_n$ that is ramified in $\widetilde{K}_\infty$ must already be ramified in $K_\infty$, since $P(\widetilde{K}_\infty)\subset P(K_\infty)$ by the definition of $\mathcal U(K_\infty,r)$. 

Fix a prime $v$ of $K$ that is ramified in $\widetilde{K}_\infty$, in particular $v\mid p$. Choose a prime $\widetilde{w}\mid v$ of $\widetilde{K}_\infty$ and let $\widetilde{I}_{\infty,\widetilde{w}}$ be the inertia subgroup of $G_{\widetilde{K}_{\infty,\widetilde{w}}}$. As in the proof of Proposition \ref{thm:controlSel}, $H^0(\widetilde{I}_{\infty,\widetilde{w}}, A^-)$ is either finite or equal to $A^-$. 

Suppose $H^0(\widetilde{I}_{\infty,\widetilde{w}}, A^-)$ is finite. Then, we follow a similar strategy as in the proof of [\emph{loc. cit}]. Consider $\widetilde{K}_{\infty}\in \calU(K_\infty, r_v)$, where $r_v=v_p(|H^0(I_{\infty,w},A^-)|)+1$ and $w\mid v$ is a prime of $K_\infty$.
In particular, 
\begin{align*}
    |H^1(\Gamma_{n,v_n}, H^0(\widetilde{I}_{\infty,\widetilde{w}},A^-))|= |H^0(\Gamma_{n,v_n}, H^0(\widetilde{I}_{\infty,\widetilde{w}},A^-))| \leq|H^0(\Gamma_{n,v_n},H^0(I_{\infty,w},A^-))| \leq |H^0(I_{\infty,w},A^-)|
\end{align*} for each $n \leq v_p(|H^0(I_{\infty},A^-)|)+1$. So, $|H^1(\Gamma_{n,v_n}, H^0(\widetilde{I}_{\infty,\widetilde{w}},A^-))|=|H^1(\Gamma_{n+1,v_{n+1}}, H^0(\widetilde{I}_{\infty,\widetilde{w}},A^-))|$ for some $n \leq v_p(|H^0(I_{\infty},A^-)|)+1$ and by Nakayama's Lemma $|H^0(\widetilde{I}_{\infty,\widetilde{w}},A^-)| \leq |H^0(I_{\infty,w},A^-)|$. This means that
\begin{align*}
    |H^1(\Gamma_{n,v_n},H^0(\widetilde{I}_{\infty,\widetilde{w}},A^-))| \leq |H^0(I_{\infty,w},A^-)|
\end{align*}  for each $n \in \mathbb{N}$. 
Therefore, we get a uniform bound on $\ker{h_{n,v}^{\widetilde{K}_{\infty}}}$ for every $\widetilde{K}_{\infty}\in \calU(K_\infty, r_v)$, where $r_v=v_p(|H^0(I_{\infty,w},A^-)|)+1$, where $w\mid v$ is a prime of $K_\infty$.

Now, suppose that $H^0(\widetilde{I}_{\infty,\widetilde{w}}, A^-)=A^-$. Choose $r>0$ such that every prime of $K_r$ that ramifies in $K_\infty$ is totally ramified in $K_\infty$. Let $\widetilde{K}_\infty\in \calU(K_\infty, r+1)$, then every prime of $\widetilde{K}_r$, that is ramified in $\widetilde{K}_\infty$ is  totally ramified. Following the arguments in Case 2 of the proof of Proposition \ref{thm:controlSel}, we get  $\ker{h}^{\widetilde{K}_{\infty}}_{t,v_t}=0$ for every $t\geq r+1$ and for every prime $v_t\mid v$ of $\widetilde{K}_t$. 

Now, by applying snake lemma it is easy to see that for every $\widetilde{K}_\infty\in \calU(K_\infty,r_0)$, where $r_0>\text{max} \{r_v| v \text{ is a prime of } K \text{ ramified in } K_\infty\}$, the kernel and cokernel of the map $f_n^{\widetilde{K}_{\infty}}$ are finite and uniformly bounded for every $n>m$. Here, $m>0$ is chosen so that every prime of $K_m$ that ramifies in $K_\infty$ is totally ramified.
\end{proof}
\begin{proposition}\label{control strsel}
    Let $K$ be a number field and let $K_\infty/K$ be a $\Z_p$-extension. Assume the hypotheses of Lemma~\ref{thm:Fukudanbd}. Fix $\widetilde K_\infty\in\calU(K_\infty,r)$, with $r$ as in that lemma, write $\widetilde K_n$ for its $n$-th layer, and set $\widetilde\Gamma_n:=\Gal(\widetilde K_\infty/\widetilde K_n)$. Then the kernel of the restriction map 
\[
\Sel^{\mathrm{str}}(A/\widetilde{K}_n)
   \overset{f^*_n}\longrightarrow
   \Sel^{\mathrm{str}}(A/\widetilde{K}_\infty)^{\widetilde\Gamma_n}
\]
is finite and uniformly bounded as $n$ varies. Further, there exists $m>0$ such that $\coker(f^*_n)$ is bounded independent of $n$ for all $n \ge m$.
\end{proposition}
\begin{proof}
   Consider the following diagram,
    \[\begin{tikzcd}
	0 & {\Sel^{\mathrm{str}}(A/\widetilde{K}_{\infty})^{\widetilde\Gamma_n}} & {\Sel^{\Gr}(A/\widetilde{K}_{\infty})^{\widetilde\Gamma_n}} & {\underset{w \mid p}\prod H^1(D_{\infty,w}/I_{\infty,w},({A^-})^{I_{\infty,w}})^{\widetilde\Gamma_{n,v_n}}} \\
	0 & {\Sel^{\mathrm{str}}(A/\widetilde{K}_{n})} & {\Sel^{\Gr}(A/\widetilde{K}_{n})} & {\underset{v \mid p}\prod H^1(D_{n,v_n}/I_{n,v_n},({A^-})^{I_{n,v_n}})}
	\arrow[from=1-1, to=1-2]
	\arrow[from=1-2, to=1-3]
	\arrow[from=1-3, to=1-4]
	\arrow[from=2-1, to=2-2]
	\arrow["{f^*_n}"', from=2-2, to=1-2]
	\arrow[from=2-2, to=2-3]
	\arrow["{f_n}"', from=2-3, to=1-3]
	\arrow[from=2-3, to=2-4]
	\arrow["{h^*_n}"', from=2-4, to=1-4]
\end{tikzcd}\]

By Proposition~\ref{thm:controlSel}, the kernel of the restriction map $f_n$ is finite for every $n$. Hence, the kernel of $f_n^*$ is also finite for all $n$. 
Let $v_n$ be a prime of $\widetilde{K}_n$ that ramifies in $\widetilde{K}_\infty$. Then necessarily $v_n\mid p$. The kernel of the rightmost vertical map is isomorphic to
$H^1\left(\widetilde\Gamma_{n,v_n},\bigl((A^-)^{I_{n,v_n}}\bigr)^{D_{\infty,w}/I_{\infty,w}}\right),$
where $w$ is a prime above $v_n$. Since, by assumption, $H^0(I_{\infty,w},A^-)$ is finite, it follows that $\ker(h_n^*)$ is finite and bounded by the order of $H^0(I_{\infty,w},A^-)$. 
The remainder of the proof now follows by arguments analogous to those used in Lemma~\ref{thm:Fukudanbd}.

%Now if $(D_{\infty,v}/I_{\infty,v})$ acts non trivially on $(A^-)^{I_{\infty,w}}$,    since by assumption $H^0(I_{\infty,w},A^-)$ is finite so $H^0(D_{\infty,v}/I_{\infty,v}, (A^-)^{I_{\infty,w}})$ is finite, and consider $B=H^0(D_{\infty,w}/I_{\infty,w}, (A^-)^{I_{\infty,w}})$ then we have $H^1(\Gamma_{n,v},B) \cong \frac{B}{(\gamma-1)B}$, where $\gamma$ is the topological generator of $\Gamma_{n,v}$. Which shows that $\ker(h_n^*)$ is bounded by $B$ and hence finite.\\
%Now, if $D_{\infty,w}/I_{\infty,w}$ acts trivially on $A^-$ and since by assumption the cyclotomic character acts via non-negative powers on $A^-$. Hence, the action of $\Gamma_{n,v}$ cannot be trivial on $A^-$. Now consider a finite extension $K_n$ of $K$ such that every prime $v_n$ of $K_n$ ramified in $K_{\infty}$  is totally ramified. Then, using a similar proof as in Proposition \ref{thm:controlSel}, we get our desired result.
\end{proof}
\begin{theorem}\label{prop:SelFukuda}
    Let $K$ be a number field and $K_\infty$ be a $\Z_p$-extension of $K$. We keep the setting and hypotheses of Lemma \ref{thm:Fukudanbd}. Let $* \in \{\Gr, \str\}$.
    Suppose that $\Sel^{*}(A/K_\infty)^\vee$ is a torsion $\Lambda$-module. Then there exist a neighbourhood $U=\cU(K_\infty,r)$, $r\geq 0$ of $K_\infty$ such that
     \begin{enumerate}[(a)]
         \item $\Sel^{*}(A/\widetilde{K}_\infty)^\vee$ is a torsion $\Lambda$-module for each $\widetilde{K}_\infty\in U$,
         \item $\mu(\Sel^{*}(A/\widetilde{K}_\infty)^\vee)\leq \mu(\Sel^{*}(A/K_\infty)^\vee)$ for each $\widetilde{K}_\infty\in U$,
         \item $\lambda(\Sel^{*}(A/\widetilde{K}_\infty)^\vee)\leq \lambda(\Sel^{*}(A/K_\infty)^\vee)$ for each $\widetilde{K}_\infty\in U$ whenever $\mu(\Sel^{*}(A/\widetilde{K}_\infty)^\vee)=\mu(\Sel^{*}(A/K_\infty)^\vee)$.
     \end{enumerate}
     
\end{theorem}

\begin{proof}
By Lemma~\ref{thm:Fukudanbd} and Proposition~\ref{control strsel}, the relevant inverse systems are Fukuda $\Lambda$-modules with parameters bounded uniformly on a sufficiently small neighbourhood. Applying \cite[Theorem~4.11]{KleineCanad} gives cotorsionness and the asserted bounds for $\mu$ and $\lambda$.
\end{proof}
\subsection{Iwasawa invariants of fine Selmer groups}
Recall that $V$ is a $p$-ordinary $2$-dimensional Galois representation which is unramified outside a finite set of primes of $K$ as introduced in Section \ref{preli}.  Let $S$ denote a finite set of primes of $K$ containing the primes above $p$, the archimedean primes and the primes where the representation  $V$ is ramified. And let $S_{\mathrm{ram}}$ denote the subset of $S$ consisting of primes at which $V$ is ramified.

\begin{definition}
For a $\mathbb Z_p$-extension $L_\infty/K$ and $n\in\N$, define
\begin{small}
\[
\begin{aligned}
\mathcal W(S,L_\infty,n)
:=\bigl\{M_\infty\in\calU(L_\infty,n)\mid{}&
 \forall v\in S, \text{if }
v\text{ splits completely in }L_\infty 
 \text{, then it splits completely in }M_\infty\bigr\}.
\end{aligned}
\]
\end{small}
These sets generate a topology on $\calE(K)$ finer than Greenberg's topology.
\end{definition}
This is a finer topology. A similar kind of topology was introduced in \cite[pp. 1402]{KleineCanad}, where the condition on splitting was placed on the bad primes of the Abelian variety. We will now show that the dual fine Selmer group $R_S{(A/\widetilde{K}_{\infty})}^\vee$ is a Fukuda $\Lambda$-Module. 
\begin{proposition}\label{prop:Fine SelFukuda_0}
    Let $K$ be a number field and $K_{\infty}/K$ be a $\Z_p$-extension. Moreover we assume that $H^0(K_{v},A)=0$ for each $v\in S_{\mathrm{ram}}$. There exist $ r \in \N$ such that $R_S{(A/\widetilde{K}_{\infty})}^\vee$ is a Fukuda $\Lambda$-module with parameters $(1,1,1)$ for each $\widetilde{K}_{\infty} \in \mathcal{W}(S, K_{\infty},r)$. 
\end{proposition}
\begin{proof}
 We will show that  for every $\widetilde{K}_\infty \in \mathcal{W}(S, K_{\infty},r)$, the kernel and cokernel of the map $R_S(A/\widetilde{K}_n)\overset{f_n}\lra R_S(A/\widetilde{K}_\infty)^{\Gamma_n}$ are finite and bounded uniformly for every $n>0$.
Consider the following commutative diagram:
   \[\begin{tikzcd}
	0 & {R_S(A/\widetilde{K}_{\infty})^{\Gamma_n}} & {H^1(G_S(\widetilde{K}_{\infty}),A)^{\Gamma_n}} & {\underset{v\in S}\prod K^1_v(A/\widetilde{K}_{\infty,v_n})^{\Gamma_{n,v_n}}} \\
	0 & {R_S(A/ \widetilde{K}_n)} & {H^1(G_S(\widetilde{K}_n),A)} & {\underset{v\in S}\prod K^1_v(A/\widetilde{K}_{n,v_n})}
	\arrow[from=1-1, to=1-2]
	\arrow[from=1-2, to=1-3]
	\arrow[from=1-3, to=1-4]
	\arrow[from=2-1, to=2-2]
	\arrow["{f_n}"', from=2-2, to=1-2]
	\arrow[from=2-2, to=2-3]
	\arrow["{g_n}"', from=2-3, to=1-3]
	\arrow[from=2-3, to=2-4]
	\arrow["{h_n= \underset{v}\prod h_{n,v_n}}"', from=2-4, to=1-4]
\end{tikzcd}\]
Since $\Gamma_n$ has the $p$-cohomological dimension $1$, we have $\coker(g_n)=0$ for all $n$. Note that the triviality of $H^0(K_v, A)$ implies that $H^0(G_S(K), A)=0$. Moreover, using the Nakayama lemma, we also get  $H^0(G_S(\widetilde{K}_\infty), A)=0$ for all $\widetilde{K}_\infty\in \mathcal{E}(K)$. Therefore, $\ker g_n=0$ and hence the map $g_n$ is an isomorphism. 

If $v_n$ is totally split in $K_\infty$, then for $\widetilde{K}_\infty\in \mathcal{W}(S, K_\infty, r)$, $v$ also totally splits in $\widetilde{K}_\infty$, we get $\ker(h_{n,v_n})=0$. 
For $v\in S_{\mathrm{ram}}$, note that $H^0(G_{K_v},A)=0$. It follows from the Nakayama lemma that $H^0(G_{K_{\infty,w}},A)=0$ where $w \mid v_n \mid S_{\mathrm{ram}}$ is a prime of $K_{\infty}$. 
%Otherwise, for $v\in S$, by Remark \ref{remark:voutsidepN} and the proof of Proposition \ref{thm:controlSel}, we get that $\ker(h_{n,v})=0$. 
The proposition now follows from the snake lemma.
   \end{proof}    

\begin{theorem}{\label{Thm fine sel_1}}
Suppose $K$ is a number field and $K_{\infty}/K$ is a $\Z_p$-extension. Moreover, we assume that $H^0(K_{v},A)=0$ for each $v\in S_{\mathrm{ram}}$. Let $\Sigma$ be a finite set containing the set $S$.
If $R_S(A/K_{\infty})^{\vee}$ is a $\Lambda$-torsion module, then there exists a neighbourhood $U= \mathcal{W}(S, K_{\infty},m)$ of $K_{\infty}$ such that
\begin{enumerate}[(a)]
         \item $R_\Sigma(A/\widetilde{K}_\infty)^{\vee}$ is a torsion $\Lambda$-module for each $\widetilde{K}_\infty\in U$,
         \item $\mu(R_\Sigma(A/\widetilde{K}_\infty)^{\vee})\leq \mu(R_S(A/K_\infty)^{\vee})$ for each $\widetilde{K}_\infty\in U$,
         \item $\lambda(R_\Sigma(A/\widetilde{K}_\infty)^{\vee})\leq \lambda(R_S(A/K_\infty)^{\vee})$ for each $\widetilde{K}_\infty\in U$ such that $\mu(R_S(A/\widetilde{K}_\infty)^{\vee})=\mu(R_S(A/K_\infty)^{\vee})$.
     \end{enumerate}
     \end{theorem}
     \begin{proof}
For $S\subset\Sigma$ there is an inclusion $R_\Sigma(A/\widetilde K_\infty)\subset R_S(A/\widetilde K_\infty)$.  The proof now follows by arguments similar to Theorem~\ref{prop:SelFukuda}.
\end{proof}
\subsection{Greenberg fine Selmer groups}
In this section, we study Iwasawa invariants for Greenberg fine Selmer. In the first few remarks, we give the comparison between Greenberg fine Selmer groups and the usual fine Selmer groups.
\begin{remark}\label{rem:3.12}
Assume that $V$ arises from a Galois representation attached to a cuspidal modular form of weight $k\ge 2$. Then, by \cite[Proposition 3.1]{kidwell}, for $v\in S_{\mathrm{ram}}$,  one has 
$ 
\ker \left(H^1(F_v, A) \rightarrow H^1(I_v, A)\right) = 0
$.   
The proof of \cite[Proposition 3.1]{kidwell} is based on the fact that the Frobenius element $\mathrm{Frob}_v$ does not have $1$ as an eigenvalue in $V$. 
Under this hypothesis, it follows that for all $v \notin S_{\mathrm{ram}}$, the local conditions defining $R_S(A/F)$ and $\Sel^{\mathrm{Gr}}(A/F)$ coincide. In particular, $R_S(A/F)$ is independent of the choice of $S$. Furthermore, in this setting, one has
$R^{\mathrm{Gr}}(A/K_\infty) \cong R_S(A/K_\infty)$, and hence $R_S(A/K_\infty)$ is independent of $S$. 
\end{remark}
\begin{remark}\label{rem:3.13}
 If $w$ is not split infinitely in $K$, then the quotient $G_{K_w}/I_w$ has profinite degree prime to $p$. In particular, the restriction map
$ H^1(K_w, A) \to H^1(I_w, A) $
is injective. In particular, if $K_\infty$ is a $\Z_p$-extension, where all primes of $K$ are finitely split then $R^{\Gr}(A/K_\infty)\cong R_S(A/K_\infty)$. Therefore, here also $R_S(A/K_\infty)$ becomes independent of $S$.
\end{remark}

\begin{theorem} {\label{Gr fine selfukuda}}
     Let $K$ be a number field and $K_{\infty}/K$ be a $\Z_p$-extension. Moreover, we assume that $H^0(G_{K_{v}},A)=0$ for each $v \mid p$. If $R^{\Gr}(A/K_{\infty})^{\vee}$ is $\Lambda$-torsion, then there exists a neighbourhood $U= \mathcal{W}(S, K_{\infty},m)$ of $K_{\infty}$ such that
\begin{enumerate}[(a)]
         \item $R^{\Gr}(A/\widetilde{K}_\infty)^{\vee}$ is a torsion $\Lambda$-module for each $\widetilde{K}_\infty\in U$,
         \item $\mu(R^{\Gr}(A/\widetilde{K}_\infty)^{\vee})\leq \mu(R^{\Gr}(A/K_\infty)^{\vee})$ for each $\widetilde{K}_\infty\in U$,
         \item $\lambda(R^{\Gr}(A/\widetilde{K}_\infty)^{\vee})\leq \lambda(R^{\Gr}(A/K_\infty)^{\vee})$ for each $\widetilde{K}_\infty\in U$ such that $\mu(R^{\Gr}(A/\widetilde{K}_\infty)^{\vee})=\mu(R^{\Gr}(A/K_\infty)^{\vee})$.
     \end{enumerate}
\end{theorem}
\begin{proof}
By the argument of Proposition~\ref{prop:Fine SelFukuda_0}, there exists $r$ such that $R^{\Gr}(A/\widetilde K_\infty)^\vee$ is a Fukuda $\Lambda$-module with parameters $(1,1,1)$ for every $\widetilde K_\infty\in\mathcal W(S,K_\infty,r)$. The conclusion then follows from \cite[Theorem~4.11]{KleineCanad}, exactly as in the proof of Theorem~\ref{prop:SelFukuda}.
\end{proof}

\begin{remark}
         In \cite{CS05}, the authors showed that the Iwasawa's $\mu =0$ for $K_{cyc}$ is connected to the structure of fine Selmer groups of elliptic curves [\emph{loc.cit}, Theorem 3.4]. Motivated by this observation, they further conjectured that for all elliptic curves $E$ over $K$, $R(E/K_{\mathrm{cyc}})^\vee$ is a finitely generated $\mathbb{Z}_p$-module. They refer to it as Conjecture A in [\emph{loc.cit}].
         
         If an analogue of Conjecture A holds for $K_\infty$, meaning that $R(A/K_\infty)^\vee$ is a finitely generated $\mathbb{Z}_p$-module, then it follows that an analogue of Conjecture A also holds for $R(A/\widetilde{K}_\infty)$ for every $\widetilde{K}_\infty \in U(K_\infty, m)$.
 \end{remark}
\section{Applications}{\label{example}}

We briefly recall the construction of $p$-adic $L$-functions for elliptic curves over arbitrary $\Z_p$-extensions, following \cite{Disegni}. Let $E/K$ be an elliptic curve, and let $p$ be a prime such that $E$ has ordinary (good or multiplicative) reduction at every prime $\mathfrak p \mid p$ of $K$. Fix a $\Z_p$-free quotient $\Gamma$ of the Galois group of the maximal abelian extension of $K$. The conjectural $p$-adic $L$-function $L_p^{(\Gamma)}(E)$ is an element of $\mathcal O_{\mathcal K'}[[\Gamma]]\otimes_{\mathcal O_{\mathcal K'}} \mathcal K',$ where $\mathcal K'/\Q_p$ is a finite extension containing, for every prime $\mathfrak p \mid p$ of good reduction, a root of the polynomial $P_{\mathfrak p}(X)=X^2-a_{\mathfrak p}X+\mathrm N(\mathfrak p), $
with $a_{\mathfrak p}=\mathrm N(\mathfrak p)+1-\#E(k_{\mathfrak p}).$
Here $k_{\mathfrak p}$ denotes the residue field at $\mathfrak p$ and $\mathrm N(\mathfrak p)$ its absolute norm. For each $\mathfrak p \mid p$, let $\alpha_{\mathfrak p}$ be the unique unit root of $P_{\mathfrak p}(X)$ when $E$ has good ordinary reduction at $\mathfrak p$, and set $\alpha_{\mathfrak p}=1$ (respectively, $\alpha_{\mathfrak p}=-1$) when $E$ has split (respectively, non-split) multiplicative reduction at $\mathfrak p$. We write $\Q(\alpha):=\Q\bigl((\alpha_{\mathfrak p})_{\mathfrak p\mid p}\bigr)\subseteq \mathcal K'.$\\

\phantomsection
\label{Hypothesis}
\textbf{Hypothesis $(L_p)$} \cite{Disegni}.
Assume that the complex $L$-function $L(E,s)$ and its twists $L(E,\chi,s)$ by finite-order characters $\chi$ of $\Gamma$ admit analytic continuation to the entire complex plane. Following \cite{Disegni}, the conjectural $p$-adic $L$-function
$L_p^{(\Gamma)}(E)\in \mathcal O_{\mathcal K'}[[\Gamma]]\otimes \mathcal K'$ is characterised by the interpolation formula
\begin{equation}\label{inerpolation_formula}
    \iota_\infty\iota_p^{-1}\bigl(L_p^{(\Gamma)}(E)(\chi)\bigr)
=
\prod_{\mathfrak p\mid p}
\iota_\infty e_{\mathfrak p}(\chi_{\mathfrak p})
\cdot
\frac{L(E,\iota_\infty\chi,1)}
{|D_K|^{-1/2}\Omega_E},
\end{equation}
 for every finite-order character $\chi$ of $\Gamma$. Here $\Omega_E$ denotes the Néron period of $E$, $D_K$ is the discriminant of $K$, and the local interpolation factors $e_{\mathfrak p}(\chi_{\mathfrak p})$ are defined by 
\begin{equation*}
e_{\mathfrak p}(\chi_{\mathfrak p})=
\begin{cases}
(1-\alpha_{\mathfrak p}^{-1}\chi(\mathfrak p)^{-1})
(1-\alpha'_{\mathfrak p}\chi(\mathfrak p))
& \text{if }\chi_{\mathfrak p}\text{ is unramified},\\[0.1cm]
\alpha_{\mathfrak p}^{-f_{\mathfrak p}}\tau(\chi_{\mathfrak p})
& \text{if }\chi_{\mathfrak p}\text{ is ramified of conductor }f_{\mathfrak p},
\end{cases}
\end{equation*}
where 
\begin{equation*}
\alpha'_{\mathfrak p}=
\begin{cases}
\alpha_{\mathfrak p}^{-1}, & \text{if }E\text{ has good reduction at }\mathfrak p,\\
0, & \text{if }E\text{ has multiplicative reduction at }\mathfrak p.
\end{cases}
\end{equation*}
%{\color{blue}Theorem below is true for a general Galois representation $A$. We may change the theorem to $A$.}
\begin{theorem}\label{thm:example}
Let $K$ be a number field and let $E/K$ be an elliptic curve with good ordinary reduction at $p$. Suppose that $\Sel^{\Gr}(E/K_\cyc)^\vee$ is a finitely generated torsion $\Lambda$-module with
$\mu(\Sel^{\Gr}(E/K_\cyc)^\vee)=\lambda(\Sel^{\Gr}(E/K_\cyc)^\vee)=0$, and assume that the Iwasawa Main Conjecture holds over $K_\cyc$. Then there exists $m\in\N$ such that, for every $\widetilde K_\infty\in\calU(K_\cyc,m)$,
\[
\mathrm{char}_{\Lambda}\bigl(\Sel^{\Gr}(E/\widetilde K_\infty)^\vee\bigr)
   =\bigl(L_p^{(\widetilde\Gamma)}(E)\bigr),
\]
where $\widetilde\Gamma=\Gal(\widetilde K_\infty/K)$ and $L_p^{(\widetilde\Gamma)}(E)$ is the $p$-adic $L$-function satisfying \eqref{inerpolation_formula}.
\end{theorem}
\begin{proof}
    Since  $\mu(\Sel^{\Gr}(E/K_\cyc)^\vee)=\lambda(\Sel^{\Gr}(E/K_\cyc)^\vee)=0$, therefore the characteristic ideal of  $\Sel^{\Gr}(E/K_\cyc)^\vee$ is a  unit in $\Z_p[[\Gamma]]$.  From our assumption that the Iwasawa Main Conjecture is true, we get that the $p$-adic $L$-function  is also a unit in $\Z_p[[\Gamma]]$, i.e. both the analytic $\mu$-invariant, $\mu_\ana^\Gamma$ and $\lambda$-invariant, $\lambda_\ana^\Gamma$  are zero. 
    From \cite[Theorem~4.11]{KleineCanad}, we get that there exists a neighbourhood $\calU(K_\cyc,m)$, such that for all $\widetilde{K}_\infty\in \calU(K_\cyc,m)$, $\Sel^{\Gr}(E/\widetilde{K}_{\infty})^\vee$ is $\Lambda$-torsion and $\mu(\Sel^{\Gr}(E/\widetilde{K}_{\infty})^\vee)=0=\lambda(\Sel^{\Gr}(E/\widetilde{K}_{\infty})^\vee)$. Choose and fix such a $\widetilde{K}_\infty\in \calU(K_\cyc,m)$. Therefore the characteristic ideal of $\Sel^{\Gr}(E/\widetilde{K}_{\infty})^\vee$ is also a unit in $\Z_p[[\widetilde{\Gamma}]]$. 

    Let $L_p^{\widetilde{\Gamma}}(E)(\chi)$ be the $p$-adic $L$-function for $E$ over $\widetilde{K}_{\infty}$. Identifying $\Z_p[[\widetilde{\Gamma}]]$ with $\Z_p[[T]]$, we can view $L_p^{\widetilde{\Gamma}}(E)(\chi)$ as an element $F(T)$ in  $\Z_p[[T]]$. Let $\mathbbm{1}$ be the trivial character of $\Gamma$. Now, applying the interpolation formula in equation \eqref{inerpolation_formula}, we get that $L_p^{\Gamma}(E)(\mathbbm{1})= L_p^{\widetilde{\Gamma}}(E)(\mathbbm{1})$. Further, since  $\mu_\ana^\Gamma=\lambda_\ana^\Gamma=0$, we get that $L_p^{\Gamma}(E)(\mathbbm{1})$ is  a $p$-adic unit and hence $L_p^{\widetilde{\Gamma}}(E)(\mathbbm{1}) $ is also a $p$-adic unit. Consequently, $F(0) = u$ where $u \in \mathbb{Z}_p^\times$, implying that $F(T) \in \mathbb{Z}_p[[T]]^\times$. The desired result immediately follows.
\end{proof}
\begin{remark}
    Note that, although the above theorem assumes that the algebraic $\mu$- and $\lambda$-invariants vanish, the conclusion remains valid if one instead assumes that the analytic $\mu$- and $\lambda$-invariants vanish, since the Iwasawa Main Conjecture holds in our setting.
\end{remark}

\begin{example}\label{ex:IMC}
Let $p=5$ and $K=\Q(i)$. Consider the elliptic curve $E/\Q$ of LMFDB label \cite[\href{https://www.lmfdb.org/EllipticCurve/Q/52/a/2}{52.a2}]{lmfdb},
\[
E:y^2=x^3+x-10.
\]
It has good ordinary reduction at $5$. The analytic and algebraic $\mu$- and $\lambda$-invariants of $E$ over $\Q_\cyc$ are zero by \cite[p.~6]{Greenbergvatsal}. Let $\chi$ be the quadratic character of $\Gal(K/\Q)$, and let $E^\chi/\Q$ denote the quadratic twist. This is the curve of LMFDB label \cite[\href{https://www.lmfdb.org/EllipticCurve/Q/208/c/2}{208.c2}]{lmfdb},
\[
E^\chi:y^2=x^3+x+10.
\]
Its analytic $\mu$- and $\lambda$-invariants over $\Q_\cyc$ are also zero. The results of Skinner--Urban \cite{Skinnerurban} and Kato \cite{kato} give the Iwasawa Main Conjecture for $E^\chi/\Q_\cyc$, and hence its algebraic invariants vanish as well. By \cite[Corollary~2.7]{pw}, the algebraic $\mu$- and $\lambda$-invariants of $E$ over $K_\cyc$ are zero. Moreover,
\[
L_p^\Gamma(E/K)=L_p^\Gamma(E/\Q)L_p^\Gamma(E^\chi/\Q).
\]
Thus the analytic invariants over $K_\cyc$ vanish. Both sides of the Iwasawa Main Conjecture are therefore units, so they generate the unit ideal. Hence the hypotheses of Theorem~\ref{thm:example} are satisfied.
\end{example}

\section*{Acknowledgment}
The first named author gratefully acknowledges the support received from the HRI postdoctoral fellowship. The fourth named author thanks the support received from NBHM postdoctoral fellowship.
\bibliographystyle{alpha}
\bibliography{bib}
\end{document}

%% file: canonicalheader.tex
\usepackage{amscd,amsfonts,amsmath,amssymb,bbm,dsfont,comment}
\usepackage{enumerate,epsf,fancyhdr,float,graphicx,tabularx}
\usepackage{latexsym,mathrsfs,multirow,wasysym,hhline}
\usepackage[OT2,T1]{fontenc}
\usepackage[modulo,mathlines]{lineno}
\usepackage{tikz,tikz-cd}
\usepackage{dashrule}
\usetikzlibrary{matrix}
\usepackage{ tipa }
\newtheorem*{conjecture*}{Conjecture}

\newtheorem{theorem}{Theorem}[section]

\DeclareSymbolFont{cyrletters}{OT2}{wncyr}{m}{n}\DeclareMathSymbol{\Sha}{\mathalpha}{cyrletters}{"58}
\DeclareMathSymbol{\FSha}{\mathalpha}{cyrletters}{"11}
\renewcommand{\phi}{{\varphi}}

\renewcommand{\geq}{\geqslant}
\renewcommand{\leq}{\leqslant}

\newcommand{\cyc}{\mathrm{cyc}}

\newcommand{\links}{\left(\begin{array}{cc}}
\newcommand{\rechts}{\end{array}\right)}
\newcommand{\bai}{\left[\begin{array}{cc}}
\newcommand{\dai}{\end{array}\right]}
\newcommand{\hidari}{\left(\begin{array}{c}}
\newcommand{\migi}{\end{array}\right)}

\newcommand{\N}{\mathbb{N}}

\newcommand{\Q}{\mathbb{Q}}

\newcommand{\Z}{\mathbb{Z}}

\newcommand{\gp}{{\mathfrak p}}

\newcommand{\calE}{\mathcal{E}}

\newcommand{\calK}{\mathcal{K}}

\newcommand{\calO}{\mathcal{O}}

\newcommand{\calU}{\mathcal{U}}

\DeclareMathOperator{\Gr}{Gr}

\newcommand{\gal}{\mathrm{Gal}}

\newcommand{\Gal}{\operatorname{Gal}}

\newcommand{\Hom}{\operatorname{Hom}}

\newcommand{\pr}{\operatorname{pr}}

\newcommand{\coker}{\operatorname{coker}}

\newcommand{\ord}{\operatorname{ord}}

\newcommand{\Sel}{\operatorname{Sel}}
\newcommand{\gr}{\operatorname{Gr}}

\newcommand{\unram}{\textrm{ur}}

\newcommand{\lra}{\longrightarrow}
\newcommand{\ana}{\mathrm{ana}}

\renewcommand{\contentsname}{Contents\\{\footnotesize\normalfont(A table
of contents should normally not be included)}}

\newtheorem{assumption}[theorem]{Assumption}
\newtheorem{auxiliary proposition}[theorem]{Auxiliary Proposition}

\newtheorem{corollary}[theorem]{Corollary}

\newtheorem{definition}[theorem]{Definition}

\newtheorem{example}[theorem]{Example}

\newtheorem{lemma}[theorem]{Lemma}
\newtheorem{main conjecture}[theorem]{Main Conjecture}
\newtheorem{main theorem}[theorem]{Main Theorem}
\newtheorem{modesty proposition}[theorem]{Modesty Proposition}

\newtheorem{open problem}[theorem]{Open Problem}

\newtheorem{proposition}[theorem]{Proposition}

\newtheorem{remark}[theorem]{Remark}

\newtheorem{convergence lemma}[theorem]{Convergence Lemma}
\newtheorem{corrected lemma}[theorem]{Corrected Lemma}
\newtheorem{growth lemma}[theorem]{Growth Lemma}
\newtheorem{coefficient lemma}[theorem]{Integrality Lemma}
\newtheorem{interpolation lemma}[theorem]{Interpolation Lemma}
\newtheorem{kernel lemma}[theorem]{Kernel Lemma}
\newtheorem{limit lemma}[theorem]{Limit Lemma}
\newtheorem{tandem lemma}[theorem]{Modesty Lemma}
\newtheorem{zero-finding lemma}[theorem]{Zero-Finding Lemma}

\usepackage{hyperref}

\hypersetup{
	colorlinks=true,
	linkcolor=blue,
	filecolor=magenta,      
	urlcolor=blue,
	pdftitle={Overleaf Example},
	pdfpagemode=FullScreen,
}